\documentclass[11pt]{article}
\usepackage{amsmath,amssymb, amscd}

\newtheorem{propo}{{\bf Proposition}}[section]
\newtheorem{coro}[propo]{{\bf Corollary}}
\newtheorem{lemma}[propo]{{\bf Lemma}} \newtheorem{theor}[propo]{{\bf
Theorem}} \newtheorem{ex}{{\sc Example}}[section]

\newenvironment{proof}{{\bf Proof.}}{$\Box$}
\def\R{{\mathbb R}}

\begin{document}

\vspace*{1.0in}

\begin{center} MAXIMAL SUBALGEBRAS AND CHIEF FACTORS OF LIE ALGEBRAS 
\end{center}
\bigskip

\begin{center} DAVID A. TOWERS 
\end{center}
\bigskip

\begin{center} Department of Mathematics and Statistics

Lancaster University

Lancaster LA1 4YF

England

d.towers@lancaster.ac.uk 
\end{center}
\bigskip

\begin{abstract}
This paper is a continued investigation of the structure of Lie algebras in relation to their chief factors, using concepts that are analogous to corresponding ones in group theory. The first section investigates the structure of Lie algebras with a core-free maximal subalgebra. The results obtained are then used in section two to consider the relationship of two chief factors of $L$ being $L$-connected, a weaker equivalence relation on the set of chief factors than that of being isomorphic as $L$-modules. A strengthened form of the Jordan-H\"older Theorem in which Frattini chief factors correspond is also established for every Lie algebra. The final section introduces the concept of a crown, a notion introduced in group theory by Gasch\"utz, and shows that it gives much information about the chief factors
\par 

\par

\end{abstract}

\section{Primitive algebras}
Throughout $L$ will denote a finite-dimensional Lie algebra over a field $F$. The symbol `$\oplus$' will denote an algebra direct sum, whilst `$\dot{+}$' will denote a direct sum of the underlying vector space structure alone. If $U$ is a subalgebra of $L$ we define $U_L$, the {\em core} (with respect to $L$) of $U$ to be the largest ideal of $L$ contained in $U$. We say that $U$ is {\em core-free} in $L$ if $U_L = 0$. We shall call $L$ {\em primitive} if it has a core-free maximal subalgebra. The {\em centraliser} of $U$ in $L$ is $C_L(U)= \{x \in L : [x,U]=0 \}$. Then we have the following characterisation of primitive Lie algebras.
\bigskip

\begin{theor}\label{t:class} 
\begin{itemize}
\item[1.]  A Lie algebra $L$ is primitive if and only if there exists a subalgebra $M$ of $L$ such
that $L = M+A$ for all minimal ideals $A$ of $L$.
\item[2.]  Let $L$ be a primitive Lie algebra. Assume that $U$ is a core-free maximal subalgebra
of $L$ and that $A$ is a non-trivial ideal of $L$. Write $C = C_L(A)$.
Then $C \cap U = 0$. Moreover, either $C = 0$ or $C$ is a minimal ideal of $L$.
\item[3.]  If L is a primitive Lie algebra and $U$ is a core-free maximal subalgebra of $L$,
then exactly one of the following statements holds:
\begin{itemize}
\item[(a)] $Soc(L) = A$ is a self-centralising abelian minimal ideal of $L$ which is complemented by $U$; that is, $L = U \dot{+} A$.
\item[(b)] $Soc(L) = A$ is a non-abelian minimal ideal of $L$ which is
supplemented by $U$; that is $L = U+A$. In this case $C_L(A) = 0$.
\item[(c)] $Soc(L) = A \oplus B$, where $A$ and $B$ are the two unique minimal ideals of $L$ and both are complemented by $U$; that is, $L = A \dot{+} U = B \dot{+} U$. In this case $A = C_L(B)$, $B = C_L(A)$, and $A$, $B$ and $(A+B) \cap U$ are non-abelian isomorphic algebras.
\end{itemize}
\end{itemize}
\end{theor}
\begin{proof} \begin{itemize} \item[1.]  If $L$ is primitive and $U$ is a core-free maximal subalgebra then it is clear that $L = U+A$ for every minimal ideal $A$ of $L$. Conversely, if there exists a subalgebra $M$ of $L$ such that $L = M+A$ for every minimal ideal $A$ of $L$ and $U$ is a maximal subalgebra of $L$
such that $M \subseteq U$, then $U$ cannot contain any minimal ideal of $L$,
and therefore $U$ is a core-free maximal subalgebra of $L$.
\item[2.] Since $U$ is core-free in $L$, we have that $L = U+A$. Since $A$ is an ideal of $L$,
then $C$ is an ideal of $L$ and then $C \cap U$ is an ideal of $U$. Since $[C \cap U,A]=0$, we have that $C \cap U$ is an ideal of $L$. Therefore $C \cap U = 0$.
\par

If $C \neq 0$, consider a minimal ideal $X$ of $L$ such that $X \subseteq C$.
Since $X \not \subseteq U$, then $L = X+U$. But now $C = C \cap (X+U) = X+(C \cap U)=X$.
\item[3.] Let us assume that $A_1, A_2, A_3$ are three pairwise distinct minimal ideals of $L$. Since $A_1 \cap A_2 = A_1 \cap A_3 = A_2 \cap A_3 = 0$, we have that
$A_2 \oplus A_3 \subseteq C_L(A_1)$. But then $C_L(A_1)$ is not a minimal ideal of $L$,
and this contradicts $2$. Hence, in a primitive Lie algebra there exist at most two
distinct minimal ideals.
\par

Suppose that $A$ is a non-trivial abelian ideal of $L$. Then $A \subseteq C_L(A)$. Since by $2$, $C_L(A)$ is a minimal ideal of $L$, we have that $A$ is self-centralising. Thus, in a primitive Lie algebra there exists at most one abelian minimal ideal of $L$. Moreover, $L = A+U$ and $A$ is
self-centralising. Then $A \cap U = C_L(A) \cap U = 0$.
\par

If there exists a unique minimal non-abelian ideal $A$ of $L$, then
$L = A+U$ and $C_L(A) = 0$.
\par

If there exist two minimal ideals $A$ and $B$, then $A \cap B = 0$
and then $B \subseteq C_L(A)$ and $A \subseteq C_L(B)$. Since $C_L(A)$ and $C_L(B)$ are minimal
ideals of $L$, we have that $B = C_L(A)$ and $A = C_L(B)$. Now $A \cap U =
C_L(B) \cap U = 0$ and $B \cap U = C_L(A) \cap U = 0$. Hence $L = A \dot{+} U = B \dot{+} U$.
Since $A = C_L(B)$, it follows that $B$ is non-abelian. Analogously we have
that $A$ is non-abelian. Furthermore, we have $A+((A+B) \cap U) = A+B = B+((A+B) \cap U)$. Hence 
\[
A \cong \frac{A}{A \cap B} \cong \frac{A+B}{B} \cong \frac{B+((A+B) \cap U)}{B} \cong (A+B) \cap U. 
\]
Analogously $B \cong (A+B) \cap U$.
\end{itemize} 
\end{proof}
\bigskip

As in the group-theoretic case this leads to three types. A primitive Lie algebra is said to be 
\begin{itemize}
\item[1.] {\em primitive of type $1$} if it has a unique minimal ideal that is abelian;
\item[2.] {\em primitive of type $2$} if it has a unique minimal ideal that is  non-abelian; and
\item[3.] {\em primitive of type $3$} if it has precisely two distinct minimal ideals each of which is  non-abelian.
\end{itemize}

Of course, primitive Lie algebras of types $2$ and $3$ are semisimple, and those of types $1$ and $2$ are monolithic. (A Lie algebra $L$ is called {\em monolithic} if it has a unique minimal ideal $W$, the {\em monolith} of $L$.)

\begin{ex} Examples of each type are easy to find.
\begin{itemize}
\item[1.] Clearly every primitive solvable Lie algebra is of type $1$.
\item[2.] Every simple Lie algebra is primitive of type $2$.
\item[3.] If $S$ is a simple Lie algebra then $L=S \oplus S$ is primitive of type $3$ with core-free maximal subalgebra $D = \{ s+s : s \in S \}$, the diagonal subalgebra of $L$.
\end{itemize}  
\end{ex}

Let $M$ be a maximal subalgebra of $L$. Then $M/M_L$ is a core-free maximal subalgebra of $L/M_L$. We say that $M$ is
\begin{itemize}
\item[1.] {\em a maximal subalgebra of type $1$} if $L/M_L$ is primitive of type $1$;
\item[2.] {\em a maximal subalgebra of type $2$} if  $L/M_L$ is primitive of type $2$; and
\item[3.] {\em a maximal subalgebra of type $3$} if  $L/M_L$ is primitive of type $3$.
\end{itemize}

\begin{lemma} Let $L$ be a non-trivial Lie algebra. 
\begin{itemize}
\item[(i)] If $M$ is a maximal subalgebra of $L$, then $L/M_L$ is primitive.
\item[(ii)] If $B$ is an ideal of $L$ and $L/B$ is primitive, then $L$ has a maximal subalgebra $M$ such that $B=M_L$.
\end{itemize}
\end{lemma}
\begin{proof}
\begin{itemize}
\item[(i)] This is easy.
\item[(ii)] Let $M/B$ be a core-free maximal subalgebra of $L/B$. Then $M$ is a maximal subalgebra of $L$ and $M_L=B$.
\end{itemize}
\end{proof}
\bigskip

We say that an ideal $A$ is {\em complemented} in $L$ if there is a subalgebra $U$ of $L$ such that $L=A+U$ and $A \cap U=0$. For primitive solvable Lie algebras we have the following analogue of Galois' Theorem for groups.

\begin{theor}\label{t:galois}
\begin{itemize}
\item[1.] If $L$ is a solvable primitive Lie algebra then all core-free maximal subalgebras are conjugate.
\item[2.] If $A$ is a self-centralising minimal ideal of a solvable Lie algebra $L$, then $L$ is primitive, $A$ is complemented in $L$, and all complements are conjugate.
\end{itemize}
\end{theor}
\begin{proof}
\begin{itemize}
\item[1.] This is \cite[Lemma 3]{conj}.
\item[2.] This follows easily from \cite[Lemma 1.5]{perm} and $1$.
\end{itemize}
\end{proof}
\bigskip

 The {\em Frattini ideal} of $L$, $\phi(L)$, is the  core of intersection of the maximal subalgebras of $L$. We say that $L$ is {\em $\phi$-free} if $\phi(L)=0$. Then we have the following characterisation of primitive Lie algebras of type $1$.

\begin{theor}\label{t:type1} Let $L$ be a Lie algebra over a field $F$.
\begin{itemize}
\item[1.] $L$ is primitive of type $1$ if and only if $L$ is monolithic, with abelian monolith $W$, and $\phi$-free.
\item[2.] If $F$ has characteristic zero, then $L$ is primitive of type $1$ if and only if $L=W \ltimes (C \oplus S)$, where $W$ is the abelian monolith of $L$, $C$ is an abelian subalgebra of $L$, every element of which acts semisimply on $W$, and $S$ is a Levi subalgebra of $L$.
\item[3.] If $L$ is solvable, then $L$ is primitive if and only if it has a self-centralising minimal ideal $A$. 
\end{itemize}
\end{theor}
\begin{proof}
\begin{itemize}
\item[1.] If $L$ is primitive of type $1$ then it has the stated properties, by Theorem \ref{t:class}. So suppose that $L$ is monolithic with abelian monolith $W$ and $\phi$-free. Then there is a maximal subalgebra $M$ of $L$ such that $L=W \dot{+} M$. If $M_L \neq 0$ there is a minimal ideal of $L$ contained in $M_L$ and distinct from $W$, a contradiction. Hence $M_L=0$ and $L$ is primitive of type $1$.
\item[2.] This follows from $1$ and \cite[Theorem 7.5]{frat}.
\item[3.] If $L$ is solvable then it has the stated property, by Theorem \ref{t:class}. The converse follows from \cite[Lemma 1.5]{perm}.
\end{itemize} 
\end{proof}

\begin{theor}\label{t:type1and3} For a Lie algebra $L$ the following are pairwise equivalent:
\begin{itemize}
\item[1.] $L$ is primitive of type $1$ or $3$;
\item[2.] there is a minimal ideal $B$ of $L$ complemented by a subalgebra $U$ which also complements $C_L(B)$;
\item[3.] there is a minimal ideal $B$ of $L$ such that $L$ is isomorphic to the semi-direct sum $X = B \ltimes L/C_L(B)$.
\end{itemize}
\end{theor}
\begin{proof}  \begin{itemize}\item[$1 \Rightarrow 2$] : This is clear from Theorem \ref{t:class}. 
\item[$2 \Rightarrow 1$] : Since $U_L \cap B = 0$ we have $U_L \subseteq C_L(B)$. But now $U_L \cap C_L(B) = 0$ implies that $U_L = 0$.  Suppose that $M$ is a proper subalgebra of $L$ containing $U$. Then $M \cap B$ is an ideal of of $M$ and is centralised by $C_L(B)$, so $M \cap B$ is an ideal of $M + C_L(B) = L$. By the minimality of $B$ we have that $M \cap B = 0$ and $U = M$. It follows that $U$ is a core-free maximal subalgebra of $L$ and $L$ is primitive. Finally note that the minimal ideal of a primitive Lie algebra of type $2$ has trivial centraliser. 
\item[$2 \Rightarrow 3$] : Simply note that $L=B \dot{+} U$ and $U \cong L/C_L(B)$, whence the map $\theta :  B \ltimes L/C_L(B) \rightarrow L$ defined by $\theta(b+u)=b+(u+C_L(B))$ is the required isomorphism.
\item[$3 \Rightarrow 2$] : Put $C = C_L(B)$ and assume there is an isomorphism $\theta : L \rightarrow B \ltimes L/C$. Consider the following subalgebras: $B^* = \theta( \{b+C:b \in B\})$, $U^*=\theta (\{(0+(x+C) : x \in L \})$ and $C^*= \theta (\{ b+(x+C) : b+x \in C \})$. For each $b \in B$, we have that $\theta(-b+(b+C))$ is a non-trivial element of $C^*$, and so $C^* \neq 0$. It is easy to check that $B^*$ is a minimal ideal of $L$, that $C^*=C_L(B^*)$ and that $U^*$ complements $B^*$ and $C^*$.
\end{itemize}
\end{proof}
\bigskip

As usual, ${\mathcal O}(m;\underline{1})$ denotes the truncated polynomial ring in $n$ indeterminates. Then the above yields the following characterisation of primitive Lie algebras of type $3$.

\begin{coro}\label{c:type3}
\begin{itemize}
\item[1.] $L$ is primitive of type $3$ if and only if $L$ has two distinct minimal ideals $B_1$ and $B_2$, both isomorphic to $S \otimes {\mathcal O}(m;\underline{1})$, where $S$ is simple, with a common complement and such that the factor algebras $L/B_i$ are primitive of type $2$ for $i=1,2$. In this case 
\[ L \subseteq \bigoplus_{i=1}^2 (Der\,S_i \otimes {\mathcal O}(m;\underline{1}))  \oplus (Id_{S_i} \otimes W(m;\underline{1})),
\]
where $S_1 \cong S_2 \cong S$.
\item[2.] If $F$ has characteristic zero, then $L$ is primitive of type $3$ if and only if $L = S \oplus S$, where $S$ is simple.
\end{itemize}
\end{coro}
\begin{proof} \begin{itemize} \item[1.] Suppose first that $L$ is primitive of type $3$. Then $L$ has two distinct minimal ideals $B_1$ and $B_2$ which have a common complement $U$ in $L$, by Theorem \ref{t:class}. Also, $U \cong L/B_1$ and $(B_2+B_1)/B_1$ is a minimal ideal of $L/B_1$. If $x+B_1 \in C_{L/B_1}((B_2+B_1)/B_1)$ then $[b,x] \in B_1$ for all $b \in B_2$, which yields that $[x,b] \in B_1 \cap B_2 = 0$ and so $x \in C_L(B_2) = B_1$. Hence $C_{L/B_1}((B_2+B_1)/B_1)=0$, implying that $L/B_1$ is primitive of type $2$, and therefore so are $U$ and $L/B_2$. Finally $B_1$ and $B_2$ are both isomorphic to $S \otimes {\mathcal O}(m;\underline{1})$, where $S$ is simple, by \cite{block}.
\par

Conversely, suppose that $L$ has two distinct minimal ideals $B_1$ and $B_2$ with a common complement $U$ and such that the factor algebras $L/B_i$ are primitive of type $2$ for $i=1,2$. Then $U \cong L/B_i$ is primitive of type $2$ such that $Soc(L/B_i) = (B_1+B_2)/B_i$ and $C_L((B_1+B_2)/B_i) = B_i$. It follows that $C_L(B_2)=B_1$ and $C_L(B_1)=B_2$. By Theorem \ref{t:type1and3} this implies that $L$ is primitive of type 3. 
\item[2.] If $L$ is primitive of type $3$ it must be semisimple with precisely two ideals which are isomorphic to each other, and thus is as described.
\end{itemize}
\end{proof}
\bigskip

This leaves primitive Lie algebras of type $2$, which can be characterised as follows.

\begin{theor}\label{t:type2}
\begin{itemize}
\item[1.] $L$ is primitive of type $2$ if and only if 
\[ L \cong U + (S \otimes {\mathcal O}(m;\underline{1})) \subseteq (Der\,S \otimes {\mathcal O}(m;\underline{1}))  \oplus (Id_S \otimes W(m;\underline{1})),
\] 
where $S \otimes {\mathcal O}(m;\underline{1})$ is an ideal of $L$ and $S$ is simple. 
\item[2.] If $F$ has characteristic zero, then $L$ is primitive of type $2$ if and only if $L$ is simple.
\item[3.] $L$ is primitive of type $2$ if and only if there is a primitive Lie algebra $X$ of type $3$ such that $L \cong X/B$ for a minimal ideal $B$ of $L$. 
\end{itemize}
\end{theor}
\begin{proof} \begin{itemize} 
\item[1] and $2$ are clear from Theorem \ref{t:class}.
\item[3.] Suppose that $L$ is primitive of type $2$, and let $D$ be the unique minimal ideal of $L$. Then $D$ is non-abelian and $C_L(D) =0$, so the semi-direct sum $X=D \ltimes L$ is primitive of type $3$, by Theorem \ref{t:type1and3}. Clearly, if $B = \{ b+0 : b \in D \}$, then $X/B \cong L$. The converse follows easily from Corollary \ref{c:type3} and Theorem \ref{t:class}.
\end{itemize}
\end{proof}
\bigskip

A special case of the above occurs as follows. We will call a Lie algebra $L$ {\em almost simple} if it is a subalgebra of Der\,$S$ for some simple subalgebra $S$ of $L$. Over a field of characteristic zero such an algebra has to be simple, as $L$ would be sandwiched between Inn\,$S (\cong S)$ and Der\,$S$, and simple algebras over such fields have no outer derivations. However, that is not the case for a field of characteristic $p$, even if it is algebraically closed and $S$ is restricted (see \cite{s-f}). It is straightforward to check that $C_{\hbox{Der}\,(S)}(S) = 0$ and thus that an almost simple Lie algebra is primitive of type $2$.

\section{Chief factors}
We say that two chief factors are {\em $L$-isomorphic}, denoted by `$\cong_L$', if they are isomorphic as $L$-modules. The {\em centraliser} of a chief factor $A/B$ is $C_L(A/B)= \{ x \in L : [x,A] \subseteq B \}$.

\begin{theor}\label{t:L-isom} Let $L$ be a Lie algebra and let $A_1/B_1$, $A_2/B_2$ be two chief factors of $L$. Then
\begin{itemize}
\item[(i)] if $A_1/B_1$ and $A_2/B_2$ are $L$-isomorphic, $C_L(A_1/B_1) = C_L(A_2/B_2)$;
\item[(ii)] if $A_1/B_1$ and $A_2/B_2$ are non-abelian, they are $L$-isomorphic if and only if $C_L(A_1/B_1) = C_L(A_2/B_2)$. 
\end{itemize}
\end{theor}
\begin{proof} (i) is clear, so suppose that $C = C_L(A_1/B_1) = C_L(A_2/B_2)$. Then $B_i \subseteq C \cap A_i$ for $i=1,2$. Since $A_i/B_i$ is non-abelian we have $A_i \not \subseteq C$, which yields that $B_i = C \cap A_i$  and $A_i/B_i$ is $L$-isomorphic to $(A_i+C)/C$ for $i=1,2$. Now $(A_1+C)/C$ is a minimal ideal of $L/C$ and $C_{L/C}((A_1+C)/C) = C$,  so $L/C$ is primitive of type $2$. But $(A_2+C)/C$ is also a minimal ideal of $L/C$, so we have that $A_1+C=A_2+C$. It follows that  $A_1/B_1$ and $A_2/B_2$ are $L$-isomorphic.
\end{proof}
\bigskip

It is easy see that \ref{t:L-isom}(ii) does not hold for abelian chief factors, as the following example shows. 
\begin{ex} Let $L=\R a+ \R b + \R c + \R x$ with mutiplication $[x,a]=a$, $[x,b]=c$, $[x,c]=-b$. Then $A_1= \R a$ and $A_2= \R b + \R c$ are minimal ideal of $L$ and $C_L(A_1) = \R a + \R b + \R c = C_L(A_2)$, but $A_1$ and $A_2$ are clearly not $L$-isomorphic.
\end{ex}

In \cite{prefrat} a strengthened form of the Jordan-H\"older Theorem in which Frattini chief factors correspond was established for solvable Lie algebras. However, the assumption of solvability is not needed as we shall show below. This assumption is used only to establish \cite[Lemma 2.1]{prefrat}, so we simply need a slightly modified version of that Lemma.

\begin{lemma}\label{l:corr} Let $A_1$, $A_2$ be distinct minimal ideals of the Lie algebra $L$. Then there is a bijection 
\[ \theta: \{ A_1, (A_1 + A_2)/A_1 \} \rightarrow \{ A_2, (A_1 + A_2)/A_2 \}
\]
such that corresponding chief factors are isomorphic as $L$-modules and Frattini chief factors correspond to one another.
\end{lemma}
\begin{proof}  Put $A = A_1 \oplus A_2$. Suppose first that $A_1$ is a Frattini chief factor. Then $A_1 \subseteq \phi(L)$. Thus $A/A_2 \subseteq \phi(L/A_2)$ and $A/A_2$ is a Frattini chief factor. If $A/A_1$ is also a Frattini chief factor, then $A/A_1 \subseteq \phi(L/A_1)$, which yields that $A \subseteq \phi(L)$, and all four factors are Frattini. In this case we can choose $\theta$ so that $\theta(A_1) = A/A_2$ and $\theta(A/A_1) = A_2$. If $A/A_1$ is not a Frattini chief factor, then nor is $A_2$, by the same argument as above, and so the same choice of $\theta$ suffices; likewise if none of the factors are Frattini chief factors.
\par

The remaining case is where $A_1$ and $A_2$ are not Frattini chief factors but $A/A_2$ is. 
The fact that $A/A_2$ is a Frattini chief factor means that every maximal subalgebra containing $A_2$ also contains $A$, and so contains $A_1$. But, since $A_1$ is not a Frattini chief factor, there is a maximal subalgebra $M$ not containing $A_1$. It follows that $A_2 \not \subseteq M$. Thus $L=M+A_1=M+A_2$. Also, $[L,M \cap A_1] = [M+A_2,M \cap A_1] \subseteq M \cap A_1$. It follows that $M \cap A_1 = 0$. Similarly $M \cap A_2 = 0$.  Put $C = A \cap M$. Then $L/A_i \cong M$ and this isomorphism maps the set of of maximal subalgebras of $L/A_i$ onto the set of maximal subalgebras of $M$. Since $A/A_2$ is a Frattini chief factor, every maximal subalgebra of $L$ containing $A_2$ contains $A$, so every maximal subalgebra of $M$ contains $C$. It follows that every maximal subalgebra of $L$ which contains $A_1$ also contains $C+A_1$; that is, $A/A_1$ is a Frattini chief factor of $L$. So we can choose $\theta$ so that $\theta(A_1) = A_2$ and $\theta(A/A_1) = A/A_2$.  
\end{proof}
\bigskip

Then the following result follows exactly as does \cite[Theorem 2.2]{prefrat}.

\begin{theor}\label{t:jordan} Let
\[ 0 < A_1 < \ldots < A_n = L  \hspace{1in} (1)
\]
\[ 0 < B_1 < \ldots < B_n = L \hspace{1in} (2)
\]
be chief series for the Lie algebra $L$. Then there is a bijection between the chief factors of these two series such that corresponding factors are isomorphic as $L$-modules and such that the Frattini chief factors in the two series correspond. 
\end{theor}
\bigskip

Note that if $L$ is a primitive Lie algebra of type $3$, its two minimal ideals are not $L$-isomorphic, so we introduce the following concept. We say that two chief factors of $L$ are {\em $L$-connected} if either they are $L$-isomorphic, or there exists an epimorphic image $\overline{L}$ of $L$ which is primitive of type $3$ and whose minimal ideals are $L$-isomorphic, respectively, to the given factors. (It is clear that, if two chief factors of $L$ are $L$-connected and are not $L$-isomorphic, then they are nonabelian and there is a single epimorphic image of $L$ which is primitive of type $3$ and which connects them.)

\begin{theor}\label{t:Lcon} The relation `is $L$-connected to' is an equivalence relation on the set of chief factors.
\end{theor}
\begin{proof} The relation is clearly reflexive and symmetric, so we simply have to establish transitivity. Suppose that $A_i/B_i$ are chief factors of $L$ for $i=1,2,3$, for which $A_1/B_1$ is $L$-connected to $A_2/B_2$ and $A_2/B_2$ is $L$-connected to $A_3/B_3$. If any two are $L$-isomorphic the result is clear. So suppose that there are primitive epimorphic images images $L/K$ (with minimal ideals $H/K$ and $J/K$) and $L/T$ (with minimal ideals $P/T$ and $Q/T$) such that
\[ A_1/B_1 \cong_L H/K, \hspace{.5cm} J/K \cong_L A_2/B_2 \cong_L P/T, \hspace{.5cm} A_3/B_3 \cong_L Q/T.
\] 
Then $H=C_L(J/K)=C_L(P/T)=Q$. Now $L/H \cong (L/K)/(H/K)$, which is primitive of type $2$ by Corollary \ref{c:type3}. But
\[ \frac{P+H}{H} \cong \frac{P}{P \cap H} = \frac{P}{P \cap Q} = \frac{P}{T} \cong \frac{A_2}{B_2} \cong \frac{J}{K} = \frac{J}{H \cap J} \cong \frac{J+H}{H},
\]
so $(P+H)/H$ and $(J+H)/H$ are minimal ideals of $L/H$. It follows that $P+H=J+H$. Similarly, $L/J$ is primitive of type $2$ with unique minimal ideal $(H+J)/J$. As $(P+J)/J$ is an ideal of $L/J$ we have $H+J \subseteq P+J \subseteq P+H+J =H+J$, so $P+H=J+H=P+J$.
\par

If $P=J$ then $T=P \cap Q=J \cap H=K$ and $A_1/B_1 \cong_L A_3/B_3$, a contradiction. It follows that $P \neq J$. Let $M,N$ be maximal subalgebras of $L$ such that $K \subseteq M$, $T \subseteq N$, $M$ is a common complement of $H/K$ and $J/K$, and $N$ is a common complement of $P/T$ and $Q/T$, so $M_L=K$ and $N_L=T$. Put $W=P \cap J$ and $X=M \cap N+W$. Then $L=M+H=M+K=N+P=N+Q$ from which it follows that $M \cap W$ is an ideal of $L$ and $M \cap W=K \cap T$. Now $X=L$ implies that $M=M \cap X =M \cap (M \cap N+W)=M \cap N + M \cap W= M \cap N + K \cap T=M \cap N$, a contradiction, as $M \neq N$. Thus $X \neq L$ and
\begin{align} P+X & =P+M \cap N+W=P+M\cap N=P+T+M \cap N  \nonumber \\
 & = P+(T+M) \cap N = P+N=L.   \nonumber
\end{align} 
Similarly, $J+X=L$.
\par

But $P \cap X$ is an ideal of $X$ and $(P \cap X)/W \subseteq C_{L/W}(J)$, so $P \cap X$ is an ideal of $J+X=L$. Since $P/W \cong (P+J)/J = (H+J)/J$ is a chief factor of $L$ and $W \subseteq P \cap X \subseteq P$, we have $P \cap X=P$ or $P \cap X=W$. The former implies that $P \subseteq X$, which in turn yields that $X=L$, a contradiction. Thus we have $P \cap X = W$. Similarly $J \cap X=W$. Finally, $P/W$ is a minimal ideal of $L/W$ and $C_{L/W}(P/W)=J/W$, and so $L/W$ is primitive of type $3$ and $L$-connects $A_1/B_1$ and $A_3/B_3$.
\end{proof}
\bigskip

Let $A/B$ be a chief factor of $L$.  We say that $A/B$ is a {\em Frattini} chief factor if $A/B \subseteq \phi(L/B)$. If there is a subalgebra $M$ such that $L=A+M$ and $B \subseteq A \cap M$, we say that $A/B$ is a {\em supplemented} chief factor of $L$, and that $M$ is a {\em supplement} of $A/B$ in $L$. If $A/B$ is a non-Frattini chief factor of $L$ then $A/B$ is supplemented by a maximal subalgebra $M$ of $L$.
\par

If $A/B$ is a chief factor of $L$ supplemented by a subalgebra $M$ of $L$ and $A \cap M=B$ then we say that $A/B$ is a {\em complemented} chief factor of $L$, and $M$ is a {\em complement} of $A/B$ in $L$. When $L$ is solvable it is easy to see that a chief factor is Frattini if and only if it is not complemented. 

\begin{propo}\label{p:mono} Let $A/B$ be a chief factor of the Lie algebra $L$ supplemented by the maximal subalgebra $M$, and let $C=C_L(A/B)$. Then
\begin{itemize}
\item[(i)] $(A+M_L)/M_L$ is a minimal ideal of the primitive Lie algebra $L/M_L$;
\item[(ii)] if $M$ is of type $1$ or $3$ then each chief factor of $L$ supplemented by $M$ is complemented by $M$;
\item[(iii)] if $A/B$ is abelian then $L/M_L$ is of type $1$ and is isomorphic to the semidirect sum $A/B \rtimes L/C$; and
\item[(iv)] if $A/B$ is non-abelian then $L/C$ is primitive of type $2$, Soc$(L/C) \cong_L A/B$, and, if $K$ is a maximal subalgebra supplementing $(A+C)/C$ in $L$, then $K$ is also a supplement to $A/B$ in $L$ and $K_L =C$.
\end{itemize}
\end{propo}
\begin{proof} \begin{itemize} \item[(i)] It is clear that $L/M_L$ is primitive and $B=A \cap M_L$. Suppose that $S$ is an ideal of $L$ with $M_L \subseteq S \subseteq A+M_L$. Then $S=(A+M_L) \cap S=A \cap S+M_L$ and $B \subseteq A \cap S \subseteq A$. Then either $A \cap S=B$, in which case $S=M_L$, or $A \cap S=A$, in which case $S=A+M_L$. Hence $(A+M_L)/M_L$ is a minimal ideal of $L/M_L$.
\item[(ii)] If $M$ is of type $1$ or $3$ then $(A+M_L)/M_L$ is a minimal ideal of $L/M_L$, which is primitive of type $1$ or $3$ and so $M \cap (A+M_L)=M_L$. But then $M \cap A=M_L \cap A=B$.
\item[(iii)] If $A/B$ is abelian then $L/M_L$ is primitive of type $1$, in which case $C=A+M_L$ and $M/M_L \cong L/C$. It follows that $L/M_L$  is isomorphic to the semidirect sum $A/B \rtimes L/C$.
\item[(iv)] If $A/B$ is non-abelian then two possibilities arise. If $C=M_L$ then $L/M_L$ is primitive of type $2$ and Soc$(L/C)=(A+C)/C \cong_L A/B$. If $M_L \subset C$, then $L/M_L$ is primitive of type $3$ with minimal ideals $(A+M_L)/M_L$ and $C/M_L$. In this case $L/C$ is primitive of type $2$ and Soc$(L/C) =(A+C)/C \cong_L A/B$.
\par

In either case, let $K$ be a maximal subalgebra supplementing $(A+C)/C$ in $L$. Then $L=A+K$ and $B=A \cap B=A \cap K_L$. Hence $K$ is also a supplement of $A/B$ in $L$ and $K_L=C$.
\end{itemize}
\end{proof}
\bigskip

In view of the above result, for any chief factor $A/B$ of $L$ we define the primitive algebra associated with $A/B$ in $L$ to be 
\begin{itemize}
\item[(i)] the semidirect sum $A/B \rtimes (L/C_L(A/B))$ if $A/B$ is abelian, or
\item[(ii)] the factor algebra $L/C_L(A/B)$ if $A/B$ is non-abelian.
\end{itemize}
Let $A/B$ be a supplemented chief factor of $L$ for which $M$ is a maximal subalgebra of $L$ supplementing $A/B$ in $L$ such that $L/M_L$ is monolithic and primitive. Note that Proposition \ref{p:mono} (iii) and (iv) show that such an $M$ exists; we call $M$ a {\em monolithic maximal subalgebra} supplementing $A/B$ in $L$. We say that the chief factor Soc$(L/M_L) = (A+M_L)/M_L$ is the {\em precrown} of $L$ associated with $M$ and $A/B$, or simply, a precrown of $L$ associated with $A/B$.
\par

If $A/B$ is a non-abelian chief factor of $L$, then for each maximal subalgebra $M$ of $L$ supplementing $A/B$ in $L$ such
that $L/M_L$ is a monolithic and primitive, we have that $M_L = C_L(A/B)$. Therefore the unique precrown of $L$ associated with $A/B$ is
$$ \hbox{Soc}(L/M_L) = \frac{A+M_L}{M_L} = \frac{A+ C_L(A/B)}{C_L(A/B)}. $$
However, if $A/B$ is a complemented abelian chief factor of $L$ and $M$ is a complement of $A/B$ in $L$, then the precrown of $L$ associated with $M$ and $A/B$ is
$$ \hbox{Soc}(L/M_L) = \frac{A+M_L}{M_L} = C_{L/M_L} \left( \frac{A+M_L}{M_L} \right) = \frac{C_L(A/B)}{M_L}.$$
This raises the question of how many different precrowns are associated with a particular abelian chief factor. For solvable algebras the answer is given by the following result.

\begin{propo}\label{p:conj} Let $A/B$ be a complemented chief factor of a solvable Lie algebra $L$ over a field $F$ of characteristic $p$, and suppose further that $L^2$ has nilpotency class less than $p$. Then the map which assigns to each conjugacy class of complements of $A/B$ in $L$, \{exp(ad\,$a)(M) : a \in L \}$ say, the common core $M_L$ of its elements, induces a bijection between the set of all conjugacy classes of complements of $A/B$ in $L$ and the set of all ideals of $L$ which complement $A/B$ in $C_L(A/B)$. 
\par

Hence there is a bijection between the set of precrowns of $L$ associated with $A/B$ and the set of all conjugacy classes of complements of $A/B$ in $L$.
\end{propo}
\begin{proof} Put $C = C_L(A/B)$, and let $M$ be a maximal subalgebra of $L$ such that $L=A+M$ and $A \cap M=B$. Put $N=C \cap M$. Then $N$ is an ideal of $L$ such that $C = A+N$ and $A \cap N = B$. Then $(A+N)/N \cong_L A/B$ and $(A+N)/N$ is a self-centralising minimal ideal of $L/N$. By Theorem \ref{t:galois}, $(A+N)/N$ is complemented in $L/N$ and all complements are conjugate. If $M/N$ is one of these complements, then $N = M_L$. Hence the map is surjective.
\par

Let $M$ and $S$ be two complements of $A/B$ in $L$ such that $N = M_L = S_L$. Then $L/N$ is solvable and primitive such that and $S/N$, $M/N$ are complements of Soc$(L/N) = (A+N)/N$. By \cite[Theorem 5]{conj}, there exists an
element $a \in A$ such that exp(ad\,$a)(S) = M$. Hence the correspondence is injective.
\par

The maximal subalgebras in a conjugacy class have a common core, by \cite[Theorem 4]{conj}. Finally observe that, since $A/B$ is abelian, the precrowns of $L$ associated with $A/B$ have a common numerator $C_L(A/B)$ and different denominators $M_L$, one for each conjugacy class of complements of $A/B$ in $L$.
\end{proof}

\begin{propo}\label{p:precrown} Let $A_i/B_i$, $i=1,2$, be two supplemented chief factors of $L$ that are $L$-connected, and let $C_i/R_i$ be a precrown associated with $A_i/B_i$ for $i=1,2$. Then $C_1=C_2$. 
\end{propo}
\begin{proof} If $A_i/B_i$ is abelian, then $C_1=C_L(A_1/B_1)=C_L(A_2/B_2)=C_2$.
\par

If $A_1/B_1$ and $A_2/B_2$ are nonabelian but $L$-isomorphic, they have the same precrown, by Proposition \ref{p:mono} (iv).
\par

So suppose that there is an ideal $N$ of $L$ such that $L/N$ is primitive of type $3$ with minimal ideals $E_1/N$, $E_2/N$ such that $E_1/N \cong_L A_1/B_1$ and $E_2/N \cong_L A_2/B_2$. Then $C_L(E_1/N)=E_2$ and $C_L(E_2/N)=E_1$. Hence the precrown associated with $E_1/N$ and $A_1/B_1$ is $(E_1+E_2)/E_2$, and the precrown associated with $E_2/N$ and $A_2/B_2$ is $(E_1+E_2)/E_1$.
\end{proof}
\bigskip

\section{Crowns}
Let $A/B$ be a supplemented chief factor of $L$ and put ${\mathcal J} = \{ M_L : M$  is a monolithic maximal subalgebra of  $L$ which supplements a chief factor of  $L$ which is $L$-connected to $A/B \}$. Let $R= \cap\{N : N \in {\mathcal J} \}$ and $C=A+C_L(A/B)$. Then we call $C/R$ the {\em crown} of $L$ associated with $A/B$.

\begin{lemma}\label{l:crown} Let ${\mathcal J_1} = \{N : D/N$ is a precrown associated with a chief factor $L$-connected to $A/B \}$, ${\mathcal J_2} = \{M_L : M$ is a maximal subalgebra of $L$ supplementing a chief factor $L$-connected to $A/B \}$, and ${\mathcal J_3} = \{M_L : M$ is a maximal subalgebra of $L$ supplementing a chief factor $L$-isomorphic to $A/B \}$. Then 
\[ \bigcap \{N: N \in {\mathcal J} \} = \bigcap \{N: N \in {\mathcal J_1} \} = \bigcap \{N: N \in {\mathcal J_2} \} = \bigcap \{N: N \in {\mathcal J_3} \}.
\] 
\end{lemma}
\begin{proof} This follows straightforwardly from Proposition \ref{p:mono}.  
\end{proof}

\begin{theor}\label{t:crown} Let $C/R$ be the crown associated with the supplemented chief factor $A/B$ of $L$. Then $C/R= Soc(L/R)$. Furthermore
\begin{itemize}
\item[(i)] every minimal ideal of $L/R$ is a supplemented chief factor of $L$ which is $L$-connected to $A/B$, and
\item[(ii)] no supplemented chief factor of $L$ above $C$ or below $R$ is $L$-connected to $A/B$.
\end{itemize}
In other words, there are $r$ ideals $A_1, \ldots , A_r$ of $L$ such that
\[ C/R = A_1/R \oplus \ldots \oplus A_r/R
\]
where $A_i/R$ is a supplemented chief factor of $L$ which is $L$-connected to $A/B$ for $i=1, \ldots ,r$ and $r$ is the number of supplemented chief factors of $L$ which are $L$-connected to $A/B$ in each chief series for $L$. Moreover, $\phi(L/R)=0$.
\end{theor}
\begin{proof} Let $R=N_1 \cap \ldots \cap N_r$ where $C/N_i$ are the precrowns associated with chief factors that are $L$-connected to $A/B$ and $r$ is minimal with respect to this property.  Then
\[ \theta : \frac{C}{R} = \frac{C}{N_1 \cap \ldots \cap N_r} \rightarrow \frac{C}{N_1} \oplus \ldots \oplus \frac{C}{N_r}
\] given by $\theta (c+(N_1 \cap \ldots \cap N_r))=(c+N_1, \ldots, c+N_r)$ is an $L$-monomorphism. Moreover, $C=N_i+(N_1 \cap \ldots \cap N_{i-1})$ for $i \leq r$, from the minimality of $r$, and so 
\[ \frac{N_1 \cap \ldots \cap N_{i-1}}{N_1 \cap \ldots \cap N_i} \cong_L \frac{C}{N_i}.
\]
It follows that the chain
\[ R=N_1 \cap \ldots \cap N_r \subseteq N_1 \cap \ldots \cap N_{r-1} \subseteq \ldots \subseteq N_1 \subseteq C
\] is part of a chief series for $L$ in which each chief factor is $L$-connected to $A/B$. Hence $\dim (C/R)= r\dim (A/B)$ and $\theta$ is an isomorphism.
\par

Suppose that $E/F$ is a supplemented chief factor of $L$ which is $L$-connected to $A/B$ and let $M$ be a maximal subalgebra of $L$ that is a supplement of $E/F$ in $L$. Then $E \subseteq C$, by Proposition \ref{p:precrown}. However, $E \not \subseteq R$, since $R \subseteq M_L$. It follows that no supplemented chief factor of $L$ over $C$ or below $R$ is $L$-connected to $A/B$.

By Theorem \ref{t:jordan} the number of supplemented chief factors $L$-connected to $A/B$ in each chief series of $L$ is an invariant of $L$ and coincides with the length of any section of chief series between $R$ and $C$.
\par

Next suppose that $D/R$ is a minimal ideal of $L/R$ and that $D \not \subseteq C$. Then $D \cap C=R$ and $D \subseteq C_L(A_1/R)$. But $A_1/R$ is $L$-connected to $A/B$, so $D \subseteq C_L(A_1/R) \subseteq C$, by Proposition \ref{p:precrown}.
\par

Finally, $\phi(L/R)=0$ since every minimal ideal of $L/R$ is supplemented. 
\end{proof}

\begin{coro}\label{c:Lcon} Two supplemented chief factors of $L$ define the same crown if and only if they are $L$-connected.
\end{coro}
\begin{proof} This is clear, since the crown associated with a supplemented chief factor is a direct sum of supplemented components.
\end{proof}

\begin{propo}\label{p:rad}  For any Lie algebra $L$ we have
\begin{align} & \cap \{S : \hbox{ there is a non-abelian crown } R/S \hbox{ of } L \} \nonumber \\
 = & \cap  \{M_L : M \hbox{ is maximal in } L \hbox{ and } L/M_L \hbox{ is primitive of type } 2 \} \nonumber \\
 = & \cap  \{M_L : M \hbox{ is maximal in } L \hbox{ and } L/M_L \hbox{ is primitive of type } 2 \hbox{ or } 3 \} \nonumber \\
 = & \cap  \{C : C=C_L(A/B), A/B \hbox{ a non-abelian chief factor of } L \} \nonumber \\
 = & \Gamma, \hbox{ the solvable radical of } L. \nonumber
\end{align}
\end{propo}
\begin{proof} It follows from Lemma \ref{l:crown} and Theorem \ref{t:crown} that the given intersections all yield the same ideal, $J$ say. Let $A/B$ be a chief factor of $L$ below $J$. If $A/B$ is non-abelian we have $A \subseteq J \subseteq C_L(A/B)$, a contradiction, so $J \subseteq \Gamma$. Moreover, if $R/S$ is a non-abelian crown of $L$, then $(S+ \Gamma)/S$ is a solvable ideal of $L/S$ and so is trivial, since $R/S=$ Soc$(L/S)$, by Theorem \ref{t:crown}. It follows that $\Gamma \subseteq S$ , whence $\Gamma \subseteq J$.
\end{proof}

\begin{theor}\label{t:comp} Let $L$ be a solvable Lie algebra, and let $C/R=\bar{C}$ be the crown associated with a supplemented chief factor of $L$. Then $\bar{C}$ is complemented in $\bar{L}$, and any two complements are conjugate by an automorphism of the form $1+ ad\,a$ for some $a \in \bar{C}$.
\end{theor}
\begin{proof} For simplicity we will assume that $R$ has been factored out and write the crown simply as $C$. Then $C=$ Asoc$L$ and $\phi(L)=0$, so $L=C \dot{+} U$ for some subalgebra $U$ of $L$, by \cite[Theorem 7.3]{frat}. 
\par

Let Asoc$L = A_1 \oplus \ldots \oplus A_n$ in $L$, where $A_i$ is a minimal ideal of $L$ for $i=1, \ldots, n$. Then Asoc$L=N$, where $N$ is the nilradical of $L$, by \cite[Theorem 7.4]{frat}, and $C_L(N)=N$. Now $A_i \cong_L A_j$ for each $1 \leq i,j \leq n$, and so $C_L(A_i)=N$ for $i=1, \ldots,n$. Let $D/N$ be a minimal ideal of $L/N$. Then there exists $d \in D$ which does not act nilpotently on $N$. Let $L=E_L(d)\dot{+}L_1$ be the Fitting decomposition of $L$ relative to ad\,$d$. Clearly $L_1 \subseteq N$, and $L_1$ is an ideal of $L$. Without loss of generality we can assume that $L_1=A_1 \oplus \ldots \oplus A_r$, where $r \leq n$. Since $[L_1,d]=L_1$ it follows that $[A_i,d]=A_i$ for each $i=1, \ldots,r$, whence $[N,d]=N$, since the $A_i's$ are $L$-isomorphic to each other. Thus $L_1=N$ and we can assume that $U=E_L(d)$.
\par

Let $V$ be another complement of $C$ in $L$. Then there exists $v \in V$ such that $v=d+n$ for some $n \in N$. Now $N=[N,d]$, so $n=[a,d$] for some $a \in N$. Thus $v=d+[a,d]=d(1+$ad\,$a)$. But $1+$ad\,$a$ is an automorphism of $L$, and so 
\[  U(1+ \hbox{ ad }a) = E_L(d)(1+ \hbox{ ad }a) = E_L(v) \supseteq V,
\]
since $v$ corresponds to $d$ in the isomorphism $V \rightarrow L/N \rightarrow U$. But $E_L(v)$ and $V$ are both complements to $N$ in $L$, and so $E_L(v)=V$, and the result follows.
\end{proof}
\bigskip

Let 
\[ 0 = L_0 \subset L_1 \subset \ldots \subset L_n = L  
\] be a chief series for $L$. We define the set $\mathcal{I}$ by $i \in \mathcal{I}$ if and only if $L_i/L_{i-1}$ is not a Frattini chief factor of $L$. For each $i \in \mathcal{I}$ put
\[ \mathcal{M}_i = \{ M \hbox{ is a maximal subalgebra of } L \colon L_{i-1} \subseteq M \hbox{ but } L_i \not \subseteq M\}.
\]
Then $B$ is a {\em prefrattini} subalgebra of $L$ if 
\[ B = \bigcap_{i \in \mathcal{I}} M_i \hbox{ for some } M_i \in \mathcal{M}_i.
\]
It was shown in \cite{prefrat} that the definition of prefrattini subalgebras does not depend on the choice of chief series.

\begin{theor}\label{t|:prefrat} Let $L$ be a solvable Lie algebra. Then the prefrattini subalgebras of $L$ are precisely the intersections of the complements of the crowns, one complement being taken from each crown.
\end{theor}
\begin{proof} Let $L_i/L_{i-1}$ be a non-Frattini, and hence supplemented, chief factor of $L$, and let its crown $C/R=A_1/R \oplus \ldots \oplus A_r/R$ be complemented by $K/R$. Then $M_i = A_1 \oplus \ldots \oplus \hat{A_i} \oplus \ldots \oplus A_r \dot{+} K$ (where the `hat' is over a term that is omitted from the sum) is a maximal subalgebra of $L$ such that $R \subseteq M_i$ but $A_i \not \subseteq M_i$. It is clear that if we intersect all such subalgebras over each of the crowns then we get a prefrattini subalgebra of $L$, and that this intersection is equal to the intersection of the complements $K$, one for each crown.
\par

Moreover, if $M$ is a maximal subalgebra with $R \subseteq M$ but $A_i \not \subseteq M$, then $L=A_i \dot{+} M$ and $M \cong L/A_i \cong A_1 \oplus \ldots \oplus \hat{A_i} \oplus \ldots \oplus A_r \dot{+} K$. It follows that we get all prefrattini subalgebras this way. 
\end{proof}

\begin{theor}\label{t:cap} Let $L$ be a solvable Lie algebra, $A/B$ a supplemented chief factor of $L$, $C/R$ the crown associated with $A/B$ and $K/R$ a complement for $C/R$. Then $K$ avoids every chief factor that is $L$-connected to $A/B$ and covers the rest. 
\end{theor}
\begin{proof} We have $L=C+K$, $C \cap K=R$ and $C^2 \subseteq R \subset K$. Let $E/F$ be a supplemented chief factor of $L$. Then $F+[C,E]=F$ or $E$. Suppose first that  $F+[C,E]=F$.  Then $[C,E] \subseteq F$. Note that this case must occur if $E/F$ is $L$-connected to $A/B$, since then $C = C_L(A/B) =C_L(E/F) $. But now $F+K \cap E$ is an ideal of $L$ and so $F+K \cap E=F$ or $E$. The former implies that $K \cap E \subseteq K \cap F$, whence $K \cap E = K \cap F$ and $K$ avoids $E/F$.
\par

 The latter yields that $F+K=E+K$ and $K$ covers $E/F$. In this case, if $E/F$ is $L$-connected to $A/B$ we have $R \subseteq F$ and  $E = F+K \cap E$. But $K \cap E \subseteq K \cap C=R$, from which $E \subseteq F+R=F$, a contradiction. So this case only occurs when $E/F$ is not $L$-connected to $A/B$. 
\par

The remaining possibility (which also only occurs when $E/F$ is not $L$-connected to $A/B$) is that  $F+[C,E]=E$. Then $[C,E] \subseteq F+[C,[C,E]] \subseteq F+C^2 \subseteq F+R \subseteq F+K$, giving $E \subseteq F+K$. Thus $E+K=F+K$ and $K$ covers $E/F$. 
\end{proof}

\end{document}